\documentclass{siamltex}

\usepackage[latin9]{inputenc}
\usepackage{mathrsfs}
\usepackage{amsmath}
\usepackage{amssymb}
\usepackage{esint}
\usepackage{enumerate}
\usepackage{url}

\usepackage[usenames]{color}

\newtheorem{remark}[theorem]{Remark}

\newcommand{\N}{\mathbb{N}}
\newcommand{\R}{\mathbb{R}}
\newcommand{\C}{\mathbb{C}}

\newcommand{\dx}[1][x]{\ensuremath{\,{\rm{d}} #1}}

\begin{document}

\title{Monotonicity-based inversion of the fractional Schrödinger equation I. Positive potentials}

\author{Bastian Harrach\footnotemark[2], and Yi-Hsuan Lin\footnotemark[3]}
\renewcommand{\thefootnote}{\fnsymbol{footnote}}

\footnotetext[2]{Institute for Mathematics, Goethe-University Frankfurt, Frankfurt am Main, 
Germany (harrach@math.uni-frankfurt.de)}

\footnotetext[3]{Institute for Advanced Study, The Hong Kong University Science and Technology,
Hong Kong (yihsuanlin3@gmail.com)}

\maketitle

\let\thefootnote\relax\footnotetext{\hrule \vspace{1ex} \centering This is a preprint version of a journal article published in\\
 \emph{SIAM J. Math. Anal.} \textbf{51}(4), 3092--3111, 2019
(\url{https://doi.org/10.1137/18M1166298}).
}

\begin{abstract}
We consider an inverse problem for the fractional Schr\"{o}dinger equation by using monotonicity formulas. 
We provide if-and-only-if monotonicity relations between positive bounded potentials and their associated nonlocal Dirichlet-to-Neumann maps. 
Based on the monotonicity relation, we can prove uniqueness for the nonlocal Calder\'{o}n problem in a constructive manner.
Secondly, we offer a reconstruction method for an unknown obstacles in a given domain. 
Our method is independent of the dimension and only requires the background solution of 
the fractional Schr\"{o}dinger equation.
\end{abstract}



\section{Introduction}
In this article we will give a constructive uniqueness result for the Calder\'{o}n problem for the nonlocal fractional Schr\"{o}dinger equation
and develop a shape reconstruction method to determine unknown obstacles in a given domain. 
Let
$\Omega$ be a bounded open set in $\mathbb{R}^{n}$
with $n\in \N$, and $q\in L^{\infty}_+(\Omega)$ be a potential, where $L^\infty _+(\Omega)$ consists of all $L^\infty (\Omega )$-functions with positive essential infima.
For $s\in(0,1)$, the (nonlocal) Dirichlet problem for the fractional Schrödinger equation is given by 
\begin{equation}
\begin{cases}
(-\Delta)^{s}u+qu=0 & \mbox{ in }\Omega,\\
u=F & \mbox{ in }\Omega_{e}:=\mathbb{R}^n\setminus \overline{\Omega},
\end{cases}\label{Fractional Schrodinger equation}
\end{equation}
Note that $(-\Delta)^{s}$ is a nonlocal operator as $s\in(0,1)$, so that the Dirichlet data is prescribed on the whole complement of $\Omega$ and not only on its boundary $\partial \Omega$.

The Dirichlet-to-Neumann (DtN) operator of \eqref{Fractional Schrodinger equation} 
\[
\Lambda(q):\ H(\Omega_{e})\to H(\Omega_{e})^{*}
\]
is formally given by 
\begin{equation}
\label{DtN map}
\Lambda(q)F:=(-\Delta)^{s}u|_{\Omega_{e}}, \quad \text{ where $u\in H^s(\mathbb R^n)$ solves \eqref{Fractional Schrodinger equation}}.
\end{equation}
The precise definition of $(-\Delta)^s$, the DtN map $\Lambda(q)$, and the function spaces are given in Section \ref{Section 2}.
For further properties for the nonlocal DtN maps, we also refer readers to \cite{ghosh2016calder}.

The nonlocal fractional Laplacian operator has received considerable attention for its ability to model
anomalous stochastic diffusion problems including jumps and longdistance interactions, cf., e.g., \cite{bucur2016nonlocal,ros2016nonlocal}, and the extensive list of references to applications in the introduction of \cite{di2012hitchhiks}.
Accordingly, inverse problems for the fractional Laplacian operator appear when an imaging domain is being probed by an anomalous diffusion process. The fractional diffusion model is more complicated than in the standard Brownian motion case modeled by the standard Laplacian ($s=1$).
However, recent works on inverse problems for nonlocal equations (see the references below) indicate that the inverse problem actually becomes easier to solve than in the standard Laplacian case. The present work contributes to this by showing that monotonicity-based reconstruction methods that have been developed for standard diffusion processes can also be applied to the fractional diffusion case and that the methods even become simpler and more powerful. 
 
For a list of recent works on inverse problems for nonlocal equations let us refer to  \cite{cao2017simultaneously,cao2018determining,cekic2018calder,covi2018inverse,ghosh2017calder,lai2018global}, and the review of Salo \cite{salo2017fractional}. Stability questions for the fractional Calderón problem were studied in \cite{ruland2019fractional,ruland2018exponential,ruland2018lipschitz}. Let us point out that the Calder\'on problem for the fractional Schrödinger equation was first solved by Ghosh, Salo and Uhlmann \cite{ghosh2016calder}, who showed the global uniqueness result that $\Lambda(q_1)=\Lambda(q_2)$ implies $q_1=q_2$. Remarkably, it was recently shown that uniqueness in the fractional Calderón problem already holds with a single measurement and with data on arbitrary, possibly disjoint subsets of the exterior, cf.~Ghosh, R{\"u}land, Salo, and Uhlmann \cite{ghosh2018uniqueness}. 

Uniqueness proofs for the fractional Calder\'on problem strongly rely on a \textit{strong uniqueness property},
cf.\ \cite[Theorem 1.2]{ghosh2016calder} for the fractional Schrödinger equation and \cite[Theorem 1.2]{ghosh2017calder}
for the nonlocal variable case. The strong uniqueness property states that if
$u=(-\Delta)^{s}u=0$ (for $0<s<1$) in an arbitrary open set in $\mathbb{R}^{n}$,
then $u\equiv0$ in the entire space $\mathbb{R}^{n}$ for any $n\in \N$. Note that this property is no longer true for the standard (local) Laplacian case, i.e.,
for the case $s=1$. In fact, via
this property, one can also derive a nonlocal Runge type approximation property (see \cite[Theorem 1.3]{ghosh2016calder} or Theorem
\ref{thm:runge}), which states that an arbitrary
$L^{2}$ function can be well approximated by solutions of the fractional
Schrödinger equation. Recently, Rüland and Salo \cite{ruland2019fractional}
studied the fractional Calderón problem under lower regularity assumptions and the stability results for the potentials.

The first effort of this paper is to prove uniqueness for the Calder\'{o}n problem in a constructive way. We will derive the following \textit{monotonicity formula} in Theorem \ref{Theorem equivalent monotonicity}: Let $q_0, q_1\in L^\infty _+(\Omega )$, then 
\begin{equation}\label{intro_monotonicity}
q_{1}\leq q_{0}\quad \mbox{ if and only if  } \quad \Lambda(q_1)\leq\Lambda(q_0),
\end{equation}
where $q_1\leq q_0$ is to be understood pointwise almost everywhere in $\Omega$, and 
$\Lambda(q_1)\leq\Lambda(q_0)$ is to be understood in the sense of definiteness of quadratic forms (also known as Loewner order),
cf.\ \eqref{quadratic sense} in Section \ref{Section 3}.
Similar monotonicity relations have been widely applied in the study of inverse problems, see \cite{harrach2013monotonicity,tamburrino2002new} for the origins
of the monotonicity method combined with the method of localized potentials 
\cite{tamburrino2002new,harrach2009uniqueness,harrach2010exact,harrach2012simultaneous,arnold2013unique,harrach2013monotonicity,harrach2015combining,barth2017detecting,harrach2015resolution,harrach2016enhancing,maffucci2016novel,tamburrino2016monotonicity,
garde2017comparison,garde2017convergence,harrach2017local,su2017monotonicity,ventre2017design,brander2018monotonicity,griesmaier2018monotonicity,harrach2018helmholtz,harrach2018localizing,harrach2018monotonicity,harrach2019uniqueness,seo2018learning,zhou2018monotonicity,garde2019regularized,harrach2019dimension}
for a list of recent and related works. Also note that similar arguments involving monotonicity conditions and blow-up arguments have been used in various ways in the study of inverse problems, see, e.g., \cite{alessandrini1990,ikehata1999identification,isakov1988,kohn1984determining,kohn1985determining}.

The monotonicity relation \eqref{intro_monotonicity} immediately implies a constructive global uniqueness result for the Calder\'on problem, which is our first main result in this paper, cf.\ Corollary \ref{cor:Calderon}: Any $q\in L_+^{\infty}(\Omega)$ is uniquely determined by $\Lambda(q)$ by
the following formula: For $x\in \Omega$ a.e., 
\begin{equation}
\label{eq:inf}
q(x)=\sup\left\{ \psi(x):\mbox{ $\psi$ positive (density one) simple function, } \Lambda(\psi)\leq\Lambda(q)\right\}.
\end{equation}
This shows  that one can recover an unknown potential with positive infimum by comparing the DtN map with that of simple functions. 

The second main result of this paper is on the shape reconstruction (or inclusion detection) problem for the fractional Schrödinger equation.
Let $q_0\in L^\infty_+(\Omega)$ denote a known reference coefficient, and $q_1\in L^\infty_+(\Omega)$
denote an unknown coefficient function that differs from the reference value $q_0$ in certain regions. 
We aim to find these anomalous regions (or scatterers), i.e., the support of $q_1-q_0$, from the difference of
the Dirichlet-to-Neumann-operators $\Lambda(q_1)-\Lambda(q_0)$. We will prove that this can be done \emph{without} solving the fractional Schrödinger equation
for potentials other than the reference potentials $q_0$. More precisely, we will show in Theorem \ref{thm:support_from_closed_sets} that
\begin{align}\label{support_C}
\begin{cases}
\mathrm{supp}(q_1-q_0)&=\bigcap \{ C\subseteq \Omega \text{ closed}:\ \\
& \qquad \quad \exists \alpha>0:\ 
-\alpha \mathcal{T}_C \leq \Lambda(q_1)-\Lambda(q_0)\leq  \alpha \mathcal{T}_C\},
\end{cases}
\end{align}
where $\mathcal{T}_C:=\Lambda'(q_0)\chi_C$, and $\Lambda'(q_0)$ is the Fr\'echet derivative of the DtN operator $\Lambda(q)$. 
The test operator $\mathcal{T}_C$ can be easily calculated from knowledge of the solution of the fractional Schrödinger equation with reference potentials $q_0$.
Under the additional definiteness condition that either $q_1\geq q_0$ or $q_1\leq q_0$ holds almost everywhere we will also show
that the inner support of $q_1-q_0$ fulfills 
\begin{align}\label{support_B1}
\mathrm{inn\,supp}(q_1-q_0)=
\bigcup \{B\subseteq \Omega \text{ open ball}:\ \exists \alpha>0: \Lambda(q_1) \leq \Lambda(q_0)-\alpha\mathcal{T}_B\},
\end{align}
resp., 
\begin{align}\label{support_B2}
\mathrm{inn\,supp}(q_1-q_0)=\bigcup \{B\subseteq \Omega \text{ open ball}:\ \exists \alpha>0: \Lambda(q_1) \geq \Lambda(q_0)+\alpha\mathcal{T}_B\},
\end{align}
cf.\ Theorem \ref{thm:support_from_open_balls}.

Inverse shape reconstruction problems were intensively studied in the literature, see \cite{isakov2006inverse,nakamura2015inverse} for
the comprehensive introduction and survey. There are several inclusion
detection methods, including the enclosure method, the linear sampling
method, the probe method and the factorization method, which have
been proposed to solve the inclusion detection inverse problem. These
methods strongly rely on special solutions of certain differential
equations. For example, the special solutions include the complex
geometrical optics (CGO) solution, the oscillating decaying (OD) solution
and the Wolff solution. In \cite{KS2014,KSU2011,nakamura2007identification,SW2006,SY2012reconstruction,UW2008}.
The authors used the CGO solutions to solve the inverse obstacle problems
for different mathematical models {for the isotropic problems}. However,
for the general anisotropic medium, we need to utilize more complicated
special solutions such as the OD solutions, see \cite{kuan2015enclosure,lin2014reconstruction,NUW2005(ODS),NUW2006}.
The Wolff solutions are used to solve the inverse obstacle problem
for the p-Laplace equation, see \cite{brander2015enclosure,brander2018monotonicity}.
Our monotonicity-based approach does not require to construct any special solutions to practically determine 
the inclusions via the formulas \eqref{support_C}--\eqref{support_B2}. The proof
of these formulas however relies on so-called \emph{localized potentials} \cite{gebauer2008localized} , i.e., solutions of
the fractional Schrödinger equations with very large energy on a subset of $\Omega$ and very low energy elsewhere, see
Corollary \ref{cor:localized_potentials}.

The structure of this article is given as follows. In Section \ref{Section 2}, we
provide basic reviews for the fractional Sobolev spaces, the fractional
Schrödinger equation and the nonlocal DtN map. In Section \ref{Section 3},
we demonstrate the monotonicity formulas and construct the localized potentials for our the fractional Schr\"{o}dinger equation. In Section \ref{section:calderon}, we show the converse results of the monotonicity relations, which gives if-and-only-if relations between the DtN maps and positive potentials. In addition, we provide a constructive global uniqueness proof for the Calder\'on problem by proving \eqref{eq:inf}. Finally, in Section \ref{Section 5}, we characterize the linearized nonlocal DtN map and derive the inclusion detection formulas \eqref{support_C}--\eqref{support_B2}.


\section{The Dirichlet-to-Neumann operator for the fractional Schrödinger
equation\label{Section 2}}

In this section, we briefly summarize some fundamental definitions
and notations on the fractional Schrödinger equation and the associated DtN operator. For $n\in\mathbb{N}$, we denote by 
\[
\mathscr{F},\ \mathscr{F}^{-1}:\ L^2(\mathbb{R}^{n};\mathbb{C})\to L^2(\mathbb{R}^{n};\mathbb{C})
\]
the Fourier transform and its inverse on the space of complex-valued $L^2$-functions, and let
$\mathcal{S}(\R^n;\C)$ be the Schwartz space of rapidly decreasing complex-valued functions.

For $s\in(0,1)$, the fractional Laplacian is defined by 
\[
(-\Delta)^{s}:\ \mathcal{S}(\mathbb{R}^{n};\mathbb{C})\to L^2(\mathbb{R}^{n};\mathbb{C}),\quad(-\Delta)^{s}u:=\mathscr{F}^{-1}\left(|\xi|^{2s}\mathscr{F}(u)\right).
\]
The fractional Laplacian can be extended to an operator 
\[
(-\Delta)^{s}:\ L^2(\mathbb{R}^{n};\mathbb{C})\to \mathcal{S}'(\mathbb{R}^{n};\mathbb{C})
\]
by setting 
\[
\langle (-\Delta)^{s}u, \varphi\rangle_{\mathcal{S}'\times \mathcal{S}}:=\langle u, (-\Delta)^{s} \varphi\rangle_{L^2} \quad \text{ for all }
u\in L^2(\mathbb{R}^{n};\mathbb{C}),\ \varphi\in \mathcal{S}(\mathbb{R}^{n};\mathbb{C}),
\]
and it can be shown that $(-\Delta)^{s}u$ will be real-valued for real-valued $u$ (see \cite{di2012hitchhiks,kwasnicki2017ten,stein2016singular}).
Hence, in the following we will always consider the fractional Laplacian as an operator
\[
(-\Delta)^{s}:\ L^2(\mathbb{R}^{n})\to \mathcal{S}'(\mathbb{R}^{n}),
\]
and all function spaces in this work are real-valued unless indicated otherwise.

For $0<s<1$, the $L^{2}$-based fractional Sobolev space is defined
by 
\[
H^{s}(\mathbb{R}^{n}):=\{u\in L^{2}(\mathbb{R}^{n}):\ (-\Delta)^{s/2}u\in L^{2}(\mathbb{R}^{n})\}
\]
and equipped with the scalar product 
\[
\left(u,v\right)_{H^{s}(\mathbb{R}^{n})}:=\int_{\mathbb{R}^{n}}\left((-\Delta)^{s/2}u\cdot(-\Delta)^{s/2}v+uv\right)\dx
\quad \text{ for all }u,v\in H^{s}(\mathbb{R}^{n}).
\]
It can be shown that $H^{s}(\mathbb{R}^{n})$ is a Hilbert space (see
\cite{di2012hitchhiks} for instance). Also note that $H^{s}(\mathbb{R}^{n})$
obviously contains the rapidly decreasing Schwartz functions $\mathcal{S}(\mathbb{R}^{n})$,
and a fortiori all compactly supported $C^{\infty}$-functions.

For an open set $\Omega\subseteq\mathbb{R}^{n}$ we define 
\[
H_{0}^{s}(\Omega):=\text{closure of \ensuremath{C_{c}^{\infty}(\Omega)} in \ensuremath{H^{s}(\mathbb{R}^{n})}}.
\]
Note that this space is sometimes denoted as $\widetilde{H}^{s}(\Omega)$
in the literature, but in the context of Dirichlet and Neumann boundary
value problems it seems more natural to denote this space by $H_{0}^{s}(\Omega)$.

We also define the bilinear form 
\begin{equation}
\mathscr{B}_{q}(u,w):=\int_{\mathbb{R}^{n}}(-\Delta)^{s/2}u\cdot(-\Delta)^{s/2}w\dx+\int_{\Omega}quw\,\dx,\label{bilinear form}
\end{equation}
for any $u,w\in H^{s}(\mathbb{R}^{n})$. We then have the following
variational formulation for the fractional Schrödinger equation.
\begin{lemma}
\label{lemma:equivalence} Let $q\in L^{\infty}(\Omega)$ and $f\in L^{2}(\Omega)$.
$u\in H^{s}(\mathbb{R}^{n})$ solves (in the sense of distributions)
\[
(-\Delta)^{s}u+qu=f\quad\text{ in \ensuremath{\Omega}}
\]
if and only if $u\in H^{s}(\mathbb{R}^{n})$ satisfies 
\[
\mathscr{B}_{q}(u,w)=\int_{\Omega}fw\dx \quad\text{ for all }w\in H_{0}^{s}(\Omega).
\]
\end{lemma}
\begin{proof}
Note that for $u\in H^{s}(\mathbb{R}^{n})$ we can interpret $(-\Delta)^s u$ as a distribution 
on $\Omega$, and a simple computation shows that
\begin{align*}
\lefteqn{\langle(-\Delta)^{s}u+qu-f,\psi\rangle_{\mathcal{D}'(\Omega)\times\mathcal{D}(\Omega)}}\\
  =&\int_{\mathbb{R}^{n}}(-\Delta)^{s/2}u\cdot(-\Delta)^{s/2}\psi\dx+\int_{\Omega}qu\psi\dx-\int_{\Omega}f\psi\dx=0
\end{align*}
for all test functions $\psi\in\mathcal{D}(\Omega)=C_{c}^{\infty}(\Omega)$, so that the assertion follows by continuous extension. 
\end{proof}

We now introduce the Dirichlet trace operator in abstract quotient
spaces.
\begin{lemma}
\label{lemma:dirichlet_trace} With $\Omega_{e}=\mathbb{R}^{n}\setminus\overline{\Omega}$,
we define 
\[
\gamma_{\Omega_{e}}:\ H^{s}(\mathbb{R}^{n})\to H(\Omega_{e}):=H^{s}(\mathbb{R}^{n})/H_{0}^{s}(\Omega),\quad u\mapsto u+H_{0}^{s}(\Omega).
\]
Then, for all $u,v\in H^{s}(\mathbb R^n)$, 
\[
\gamma_{\Omega_{e}}u=\gamma_{\Omega_{e}}v\quad\text{ implies that }\quad u(x)=v(x)\quad\text{ for \ensuremath{x\in\Omega_{e}} a.e.}
\]
\end{lemma}
\begin{proof}
If $\gamma_{\Omega_{e}}u=\gamma_{\Omega_{e}}v$, then $u-v\in H_{0}^{s}(\Omega)$,
so that there exists a sequence $(\phi_{k})_{k\in\mathbb{N}}\subset C_{c}^{\infty}(\Omega)$
with $\phi_{k}\to u-v$ in $H^{s}(\R^n)$. In particular, this implies that $\phi_{k}|_{\Omega_{e}}=0$
and $\phi_{k}\to u-v$ in $L^{2}(\R^n)$, so that it follows that $u|_{\Omega_{e}}=v|_{\Omega_{e}}$
as $L^{2}(\Omega_{e})$-functions. 
\end{proof}

For the sake of readability we will write $u|_{\Omega_{e}}$ instead
of $\gamma_{\Omega_{e}}u$ in the following. Also note that Lemma
\ref{lemma:dirichlet_trace} implies that for two functions $F,G\in C_{c}^{\infty}(\Omega_{e})$,
$F=G$ if and only if $F-G\in H_{0}^{s}(\Omega)$, so that we can
identify $C_{c}^{\infty}(\Omega_{e})$ with its image in the quotient
space $H(\Omega_{e})$ and thus consider $C_{c}^{\infty}(\Omega_{e})$
as a subspace of $H(\Omega_{e})$.

We can now state the following result on the solvability of the Dirichlet
problem and the definition of Neumann boundary values. 
\begin{lemma}
\label{lemma:forward_theory} Let $q\in L_{+}^{\infty}(\Omega)$. 

\begin{itemize}
	\item[(a)] For every $F\in H(\Omega_{e})$ and $f\in L^{2}(\Omega)$,
	we have that $u\in H^{s}(\mathbb{R}^{n})$ solves the Dirichlet problem
	\begin{align}
	(-\Delta)^{s}u+qu=f\quad\text{ in \ensuremath{\Omega}},\quad u|_{\Omega_{e}}=F,\label{Dirichlet BVP}
	\end{align}
	if and only if $u=u^{(0)}+u^{(F)}$, where $u^{(F)}\in H^{s}(\mathbb{R}^{n})$
	fulfills $u^{(F)}|_{\Omega_{e}}=F$ (for $F\in C_{c}^{\infty}(\Omega_{e})$ we can simply choose $u^{(F)}:=F$), 
	and $u^{(0)}\in H_{0}^{s}(\Omega)$ solves 
	\[
	\mathscr{B}_{q}(u^{(0)},w)=-\mathscr{B}_{q}(u^{(F)},w)+\int_{\Omega}fw\dx\quad\text{ for all }w\in H_{0}^{s}(\Omega).
	\]
	The Dirichlet problem \eqref{Dirichlet BVP} is uniquely solvable
	and the solution $u\in H^{s}(\mathbb{R}^{n})$ depends linearly and
	continuously on $F\in H(\Omega_{e})$ and $f\in L^{2}(\Omega)$.
	
	\item[(b)] For a solution $u\in H^{s}(\mathbb{R}^{n})$ of $(-\Delta)^{s}u+qu=0$,
	we define the Neumann exterior data $\mathcal{N}u\in H(\Omega_{e})^{*}$
	by 
	\[
	\left\langle \mathcal{N}u,G\right\rangle :=\mathscr{B}_{q}(u,v^{(G)}),\quad\text{where \ensuremath{v^{(G)}\in H^{s}(\mathbb{R}^{n})} fulfills \ensuremath{v^{(G)}|_{\Omega_{e}}=G},}
	\]
	where $H(\Omega_e)^*$ is the dual space of $H(\Omega_e)$ and $\left\langle \cdot,\cdot\right\rangle :=\left\langle \cdot,\cdot\right\rangle _{H(\Omega_{e})^{*}\times H(\Omega_{e})}$.
	Then the Dirichlet-to-Neumann operator 
	\[
	\Lambda(q):\ H(\Omega_{e})\to H(\Omega_{e})^{*},\quad F\mapsto\mathcal{N}u,
	\]
	is a symmetric linear bounded operator, where $u$ solves \eqref{Dirichlet BVP} with $f=0$.
	
\end{itemize}

\end{lemma}
\begin{proof}
Obviously $\mathscr{B}_{q}$ is a coercive, symmetric, and continuous
bilinear form on the Hilbert space $H_{0}^{s}(\Omega)$. Thus the
assertion follows from a standard application of the Lax-Milgram theorem
and the equivalence result in Lemma \ref{lemma:equivalence}. 
\end{proof}

\begin{remark}
If $u\in H^{2s}(\mathbb R^n)$  solves 
\[
(-\Delta)^{s}u+qu=0\quad\text{ in \ensuremath{\Omega}},
\]
then a computation as in the proof of Lemma \ref{lemma:equivalence}
shows that 
\[
\left\langle \mathcal{N}u,F\right\rangle =\int_{\Omega_{e}}(-\Delta)^{s}u\cdot v^{(F)}\dx
\]
where $v^{(F)}|_{\Omega_{e}}=F$. This motivates to formally write the
Neumann boundary values as 
\[
\mathcal{N}u=(-\Delta)^{s}u|_{\Omega_{e}}.
\]
Note that $\mathcal{N}u=(-\Delta)^{s}u|_{\Omega_{e}}\in L^2(\Omega_e)$ rigorously holds under the additional smoothness condition $u\in H^{2s}(\R^n)$ (not only $u\in H^s(\R^n)$) but no regularity assumptions on the open set $\Omega\subseteq \R^n$ is required. 
Alternatively, it is also possible to justify the notation $\mathcal{N}u=(-\Delta)^{s}u|_{\Omega_{e}}$ without the additional smoothness (i.e., for $u\in H^s(\R^n)$) when regularity assumptions on $\Omega$ are imposed, cf.~\cite{ghosh2016calder}. 
\end{remark}

For a more in-depth discussion on the spaces of Dirichlet and Neumann
boundary values for Lipschitz domains we refer readers to \cite{ghosh2016calder}.
For the results in this work it suffices to use the abstract quotient
space definitions given above, which also has the advantage that we
can treat arbitrary open sets $\Omega$ without any boundary regularity
assumptions. 

\section{Monotonicity relations and localized potentials for the fractional
Schrödinger equation\label{Section 3}}

In this section we will derive monotonicity and localized potentials
results for the fractional Schrödinger equation 
\[
(-\Delta)^{s}u+qu=0\quad\text{ in \ensuremath{\Omega}}.
\]

\subsection{Monotonicity relations}

We will first show that increasing the coefficient $q$ increases
the Dirichlet-Neumann-operator in the sense of quadratic forms.
\begin{lemma}
\label{Lemma for monotonicity}(Monotonicity relations) Let $n\in\mathbb{N}$,
$\Omega\subseteq\mathbb{R}^{n}$ be an open set, and $s>0$. For $j=0,1$,
let $q_{j}\in L^\infty_+(\Omega)$ and $u_{j}\in H^{s}(\mathbb{R}^{n})$
be solutions of 
\begin{equation}
\begin{cases}
(-\Delta)^{s}u_{j}+q_{j}u_{j}=0 & \mbox{ in }\Omega,\\
u_{j}|_{\Omega_{e}}=F,
\end{cases}\label{nonlocal equation for monotonicity}
\end{equation}
where $F\in H(\Omega_{e})$. Then we have the following monotonicity
relations 
\begin{equation}
\left\langle \left(\Lambda(q_{1})-\Lambda(q_{0}) \right) F,F\right\rangle \leq\int_{\Omega}(q_{1}-q_{0})|u_{0}|^{2}dx,\label{monotone relation 1}
\end{equation}
and 
\begin{equation}
\left\langle (\Lambda(q_1)-\Lambda(q_0))F,F\right\rangle \geq\int_{\Omega}\left(q_{1}-q_{0}\right)|u_{1}|^{2}dx.\label{monotone relation 2} 
\end{equation}
Moreover, we have
\begin{equation}
\left\langle (\Lambda(q_1)-\Lambda(q_0))F,F\right\rangle \geq\int_{\Omega}\dfrac{q_{0}}{q_{1}}\left(q_{1}-q_{0}\right)|u_{0}|^{2}dx\label{monotone relation 3}
\end{equation}
and
\begin{equation}
\left\langle (\Lambda(q_1)-\Lambda(q_0))F,F\right\rangle \leq\int_{\Omega}\dfrac{q_{1}}{q_{0}}\left(q_{1}-q_{0}\right)|u_{1}|^{2}dx.\label{monotone relation 4}
\end{equation}
\end{lemma}
\begin{proof}
The proof is similar to \cite[Lemma 2.1]{harrach2010exact}. Note
also that the idea goes back to Ikehata, Kang, Seo, and Sheen \cite{ikehata1998size,kang1997inverse},
and that similar results have been obtained and used in many other
works on the monotonicity and factorization method.

From the definition of the Dirichlet-Neumann-operator in Lemma~\ref{lemma:forward_theory}
we have that 
\begin{align*}
\langle\Lambda(q_0)F,F\rangle & =\mathscr{B}_{q_{0}}(u_{0},u_{0}),\text{ and}\\
\langle\Lambda(q_1)F,F\rangle & =\mathscr{B}_{q_{1}}(u_{1},u_{1})=\mathscr{B}_{q_{1}}(u_{1},u_{0})
\end{align*}
We thus obtain 
\begin{align*}
0\leq & \mathscr{B}_{q_{1}}(u_{1}-u_{0},u_{1}-u_{0})=\mathscr{B}_{q_{1}}(u_{1},u_{1})-2\mathscr{B}_{q_{1}}(u_{1},u_{0})+\mathscr{B}_{q_{1}}(u_{0},u_{0})\\
= & -\langle\Lambda(q_1)F,F\rangle+\mathscr{B}_{q_{1}}(u_{0},u_{0})\\
= & \langle(\Lambda(q_0)-\Lambda(q_1))F,F\rangle+\mathscr{B}_{q_{1}}(u_{0},u_{0})-\mathscr{B}_{q_{0}}(u_{0},u_{0})\\
= & \langle(\Lambda(q_0)-\Lambda(q_1))F,F\rangle+\int_{\Omega}(q_{1}-q_{0})|u_{0}|^{2}dx,
\end{align*}
which shows the first assertion (\ref{monotone relation 1}). Interchanging $q_{0}$ and $q_{1}$ in the above calculation also
yields \eqref{monotone relation 2}. 

In addition, when $q_j\in L^\infty _+(\Omega)$, one can see
\begin{align*}
\lefteqn{\left\langle (\Lambda(q_1)-\Lambda(q_0))F,F\right\rangle = \mathscr{B}_{q_{0}}(u_{0}-u_{1},u_{0}-u_{1})-\int_{\Omega}(q_{0}-q_{1})|u_{1}|^{2}dx}\\
& = \int_{\mathbb{R}^{n}}|(-\Delta)^{s/2}(u_{1}-u_{0})|^{2}dx+\int_{\Omega}\left(q_{0}(u_{1}-u_{0})^{2}+(q_{1}-q_{0})|u_{1}|^{2}\right)dx\\
&\geq  \int_{\Omega}\left(q_{0}(u_{1}-u_{0})^{2}+(q_{1}-q_{0})|u_{1}|^{2}\right)dx
=  \int_{\Omega}\left(q_{1}u_{1}^{2}-2q_{0}u_{1}u_{0}+q_{0}u_{0}^{2}\right)dx\\
&=  \int_{\Omega}q_{1}\left(u_{1}-\dfrac{q_{0}}{q_{1}}u_{0}\right)^{2}dx+\int_{\Omega}\left(q_{0}-\dfrac{q_{0}^{2}}{q_{1}}\right)|u_{0}|^{2}dx\\
&\geq  \int_{\Omega}\dfrac{q_{0}}{q_{1}}\left(q_{1}-q_{0}\right)|u_{0}|^{2}dx,
\end{align*}
which proves \eqref{monotone relation 3}, and interchanging $q_{0}$ and $q_{1}$ yields \eqref{monotone relation 4}. 
\end{proof}

For two functions $q_{0},q_{1}\in L_+^{\infty}(\Omega)$,
we write $q_{0}\geq q_{1}$ if $q_{0}(x)\geq q_{1}(x)$ almost everywhere
in $\Omega$. For two operators $\Lambda(q_0),\Lambda(q_1):\ H(\Omega_{e})\to H(\Omega_{e})^{*}$
we write $\Lambda(q_0)\geq\Lambda(q_1)$ if 
\begin{align}
\left\langle (\Lambda(q_0)-\Lambda(q_1))F,F\right\rangle \geq 0\mbox{ for all }F\in H(\Omega_{e}).\label{quadratic sense}
\end{align}
With respect to these partial orders, we have the following monotonicity
property of the Dirichlet-Neumann-operators.
\begin{corollary}
\label{Cor Monoton formula}For any $q_{0},q_{1}\in L^{\infty}_+(\Omega)$,
\begin{equation}
q_{0}\geq q_{1}\mbox{ implies }\Lambda(q_0)\geq\Lambda(q_1).\label{Monotonicity based reconstruction}
\end{equation}
\end{corollary}
\begin{proof}
This follows immediately from Lemma~\ref{Lemma for monotonicity}. 
\end{proof}

\subsection{Localized potentials }

In this subsection we will show the existence of localized potentials
solutions of the fractional Schr\"odinger equation that have an arbitrarily
high energy on some part of the imaging domain and an arbitrarily
low energy on another part. These localized potentials will allow
us to control the energy terms in Lemma \ref{Lemma for monotonicity}
and show a converse of the monotonicity relation \eqref{Monotonicity based reconstruction}.

For the fractional Schrödinger equation, the existence of localized
potentials is a simple consequence from the unique continuation and
Runge approximation result shown by Ghosh, Salo and Uhlmann \cite{ghosh2016calder},
see also \cite{ghosh2017calder} for further discussions and \cite{harrach2018monotonicity} for
the connection between Runge approximation properties and localized
potentials. We use the following unique continuation result from Ghosh, Salo and Uhlmann \cite{ghosh2016calder}.

\begin{theorem}
\cite[Theorem 1.2]{ghosh2016calder} \label{thm:unique_continuation}
Let $n\in \N$, and $0<s<1$. If $u\in H^{r}(\mathbb{R}^{n})$ for
some $r\in\mathbb{R}$, and both $u$ and $(-\Delta)^{s}u$ vanish
in the same arbitrary non-empty open set in $\mathbb R^n$, then $u\equiv0$ in $\mathbb{R}^{n}$. 
\end{theorem}

Ghosh, Salo and Uhlmann \cite[Theorem. 1.3]{ghosh2016calder} also showed that this unique continuation result implies the Runge-type approximation property that any $L^{2}(\Omega)$-function can be approximated by solutions of the fractional Schrödinger equation with exterior Dirichlet data supported on an arbitrarily small open set. Since our formulation
slightly differs from \cite{ghosh2016calder}, and the proof if short
and simple, we give the proof for the sake of completeness.

\begin{theorem}
\label{thm:runge} Let $n\in \N$, $\Omega\subseteq\mathbb{R}^{n}$
be an open set, $0<s<1$, $q\in L^{\infty}_+(\Omega)$, and $O\subseteq\Omega_{e}=\mathbb{R}^{n}\setminus\overline{\Omega }$
be open. For every $f\in L^{2}(\Omega)$ there exists a sequence $F_{k}\in C_{c}^{\infty}(O)$,
so that the corresponding solutions $u^{k}\in H^{s}(\mathbb{R}^{n})$
of 
\begin{align*}
(-\Delta)^{s}u^{k}+qu^{k}=0\quad\text{ in \ensuremath{\Omega}},\quad u^{k}|_{\Omega_{e}}=F_{k},
\end{align*}
fulfill that $u^{k}|_{\Omega}\to f$ in $L^{2}(\Omega)$. 
\end{theorem}
\begin{proof}
For $F\in C_{c}^{\infty}(O)$ let $S(F):=u\in H^{s}(\mathbb{R}^{n})$
denote the solution of 
\begin{align}\label{unique_continuation_Dirichlet_solution}
(-\Delta)^{s}u+qu=0\quad\text{ in \ensuremath{\Omega}},\quad u|_{\Omega_{e}}=F,
\end{align}
i.e. (see Lemma~\ref{lemma:forward_theory}) $u=u^{(0)}+F$
where $u^{(0)}\in H_{0}^{s}(\Omega)$ solves $\mathscr{B}_{q}(u^{(0)},w)=-\mathscr{B}_{q}(F,w)$
for all $w\in H_{0}^{s}(\Omega)$.

The assertion follows if we can show that the space of all such solutions
\[
\{ S(F)|_{\Omega}:\ F\in C_{c}^{\infty}(O)\}\subseteq L^{2}(\Omega)
\]
has trivial $L^{2}(\Omega)$-orthogonal complement. To that end let
\[
f\in\{S(F)|_{\Omega}:\ F\in C_{c}^{\infty}(O)\}^{\perp}\subseteq L^{2}(\Omega)
\]
and $v\in H^{s}(\mathbb R^n)$ solve 
\begin{align*}
(-\Delta)^{s}v+qv=f\quad\text{ in \ensuremath{\Omega}},\quad v|_{\Omega_{e}}=0,
\end{align*}
i.e., $v\in H_{0}^{s}(\Omega)$ and $\mathscr{B}_{q}(v,w)=\int_{\Omega}fw|_{\Omega}$dx for all $w\in H_{0}^{s}(\Omega)$.
Then for all $F\in C_{c}^{\infty}(O)$, the solution $u=S(F)$ of \eqref{unique_continuation_Dirichlet_solution} fulfills 
\begin{align*}
0 & =\int_{\Omega}f u dx=\int_{\Omega}f(u^{(0)}+F)dx=\int_{\Omega}f u^{(0)} dx=\mathscr{B}_{q}(v,u^{(0)})=-\mathscr{B}_{q}(v,F).
\end{align*}
Using Lemma \ref{lemma:equivalence} with $O$ instead of $\Omega$
this yields that 
\[
(-\Delta)^{s}v+qv=0\quad\text{ in }O.
\]
Since also $v|_{O}=0$ (cf.\ Lemma~\ref{lemma:dirichlet_trace}),
it follows from Theorem~\ref{thm:unique_continuation} that $v\equiv0$
in $\mathbb{R}^{n}$ and thus $f=0$ which proves the assertion. 
\end{proof}

Theorem~\ref{thm:runge} implies the existence of a localized potentials
sequence. 

\begin{corollary}
\label{cor:localized_potentials} Let $n\in \N$, $\Omega\subseteq\mathbb{R}^{n}$
be an open set, $0<s<1$, $q\in L^{\infty}_+(\Omega)$, and $O\subseteq\Omega_{e}=\mathbb{R}^{n}\setminus\overline{\Omega}$
be an arbitrary open set. For every measurable set $M\subseteq\Omega$ with positive measure, there exists
a sequence $F^{k}\in C_{c}^{\infty}(O)$, so that the corresponding
solutions $u^{k}\in H^{s}(\mathbb{R}^{n})$ of 
\begin{align}\label{eq:cor_loc_pot_uk}
(-\Delta)^{s}u^{k}+qu^{k}=0\quad\text{ in \ensuremath{\Omega}},\quad u^{k}|_{\Omega_{e}}=F^{k},\text{ for all }k\in \mathbb N
\end{align}
fulfill that 
\[
\int_{M}|u^{k}|^{2}\dx\to\infty\quad\text{ and }\quad\int_{\Omega\setminus M}|u^{k}|^{2}\to0 \mbox{ as }k\to \infty.
\]
\end{corollary}
\begin{proof}
Using Theorem \ref{thm:runge} there exists a sequence $\widetilde{F}^{k}$
so that the corresponding solutions $\widetilde{u}^{k}|_{\Omega}$
converge against $\dfrac{1}{|M|}\chi_{M}$ in $L^{2}(\Omega)$, and thus
\[
\|\widetilde{u}^{k}\|_{L^{2}(M)}^2=\int_{M}|\widetilde{u}^{k}|^{2}\dx\to 1,
\quad\text{ and }\quad \|\widetilde{u}^{k}\|_{L^{2}(\Omega\setminus M)}^2=\int_{\Omega\setminus M}|\widetilde{u}^{k}|^{2}\dx\to 0.
\]
Without loss of generality, we can assume for all $k\in \mathbb{N}$ that
$\widetilde{u}^{k}\not\equiv0$. Moreover, by possibly removing a sufficiently small open ball from $M$
(so that the measure of $M$ remains positive), we can assume that $\Omega\setminus M$ contains an open set.
Thus, $\|\widetilde{u}^{(k)}\|_{L^{2}(\Omega\backslash M)}>0$
follows from Theorem~\ref{thm:unique_continuation}. Setting 
\[
F^{k}:=\dfrac{\widetilde{F}^{k}}{\|\widetilde{u}^{k}\|_{L^{2}(\Omega\backslash M)}^{1/2}}  
\]
the sequence of corresponding solutions $u^{k}\in H^{s}(\mathbb{R}^{n})$ of \eqref{eq:cor_loc_pot_uk} has the desired property that
\[
\|u^{k}\|_{L^{2}(M)}^2 = \frac{\|\widetilde{u}^{k}\|_{L^{2}(M)}^2}{ \|\widetilde{u}^{k}\|_{L^{2}(\Omega\setminus M)} } \to \infty,
\quad\text{ and }\quad \|u^{k}\|_{L^{2}(\Omega\setminus M)}^2=\|\widetilde{u}^{k}\|_{L^{2}(\Omega\setminus M)}\to 0,
\]
as $k\to \infty$.
\end{proof}

The following lemma will also be useful for applying localized potentials
in next sections. 

\begin{lemma}
\label{Lemma boundedness of u_1 and u_0 } Let $n\in \N$,
$\Omega\subseteq\mathbb{R}^{n}$ be a bounded open set, $0<s<1$, and $q_{0},q_{1}\in L_{+}^{\infty}(\Omega)$.
Set $D:=\mathrm{supp}(q_{0}-q_{1})$. There exist constants $c,C>0$
so that, for all $F\in H(\Omega_{e})$, the solutions $u_{0},u_{1}\in H^{s}(\R^n)$ of 
\begin{align*}
(-\Delta)^{s}u_{j}+q_{j}u_{j}=0\quad\text{ in \ensuremath{\Omega}},\quad u_{j}|_{\Omega_{e}}=F\quad(j=0,1)
\end{align*}
fufill 
\begin{equation}
c\|u_{1}\|_{L^{2}(D)}\leq\|u_{0}\|_{L^{2}(D)}\leq C\|u_{1}\|_{L^{2}(D)}.
\label{Doubling estimates}
\end{equation}
\end{lemma}
\begin{proof}
For the difference $w:=u_{0}-u_{1}\in H_{0}^{s}(\Omega)$ we have
that 
\[
0=\mathscr{B}_{q_{0}}(u_{0},w)=\mathscr{B}_{q_{1}}(u_{1},w) \quad \text{ for all } w\in H_0^s(\Omega).
\]
With the coercivity constant $\alpha>0$ of $\mathscr{B}_{q_{1}}$
we can estimate 
\begin{align*}
\alpha\|u_{1}-u_{0}\|_{H^{s}(\Omega)}^{2} & \leq\mathscr{B}_{q_{1}}(u_{1}-u_{0},u_{1}-u_{0})=-\mathscr{B}_{q_{1}}(u_{0},u_{1}-u_{0})\\
 & =\mathscr{B}_{q_{0}}(u_{0},u_{1}-u_{0})-\mathscr{B}_{q_{1}}(u_{0},u_{1}-u_{0})\\
 & =\int_{\Omega}(q_{0}-q_{1})u_{0}(u_{1}-u_{0})\dx\\
 & \leq\|q_{0}-q_{1}\|_{L^{\infty}(\Omega)}\|u_{0}\|_{L^{2}(D)}\|u_{1}-u_{0}\|_{H^{s}(\Omega)}.
\end{align*}
Hence, 
\begin{align*}
\|u_{1}\|_{L^{2}(D)}-\|u_{0}\|_{L^{2}(D)} & \leq\|u_{1}-u_{0}\|_{L^{2}(D)}\leq\|u_{1}-u_{0}\|_{H^{s}(\Omega)}\\
 & \leq\frac{1}{\alpha}\|q_{0}-q_{1}\|_{L^{\infty}(\Omega)}\|u_{0}\|_{L^{2}(D)},
\end{align*}
which shows that 
\[
\|u_{1}\|_{L^{2}(D)}\leq C\|u_{0}\|_{L^{2}(D)}\quad\text{ with }C:=1+\frac{1}{\alpha}\|q_{0}-q_{1}\|_{L^{\infty}(\Omega)}.
\]
The other inequality follows from interchanging $q_{0}$ and $q_{1}$. 
\end{proof}

\section{Converse monotonicity relations and the nonlocal Calder\'on problem}\label{section:calderon}

This section contains the first main result of this work. We will show that the monotonicity relation between the coefficients and the Dirichlet-to-Neumann operators holds in both directions, and use this
to give a monotonicity-based, constructive proof of Ghosh, Salo and Uhlmann's uniqueness result for the Calderón problem 
for the fractional Schrödinger equation \cite{ghosh2016calder}.

\begin{theorem}
\label{Theorem equivalent monotonicity} Let $n\in \N$, $\Omega\subseteq\mathbb{R}^{n}$
be a bounded open set, and $0<s<1$. For any two potentials $q_{0},q_{1}\in L_{+}^{\infty}(\Omega)$,
we have that 
\begin{equation}
q_{0}\leq q_{1}\mbox{ if and only if }\Lambda(q_0)\leq\Lambda(q_1).\label{Monotone equivalence relation}
\end{equation}
\end{theorem}
\begin{proof}
From Corollary \ref{Cor Monoton formula}, we know that $q_{0}\leq q_{1}$
implies $\Lambda(q_0)\leq\Lambda(q_1)$. Hence, it remains to
show $\Lambda(q_0)\leq\Lambda(q_1)$ implies that $q_{0}\leq q_{1}$
a.e.\ in $\Omega$. We will prove this via contradiction and assume
that $q_{0}\leq q_{1}$ is not true a.e.\ in $\Omega$. Then there
exists $\delta>0$ and a measurable set $M\subset\Omega$ with positive
measure such that $q_{0}-q_{1}\geq\delta$ on $M$. Using the sequence
of localized potentials from Corollary \ref{cor:localized_potentials}
for the coefficient $q_{0}$, and the monotonicity inequality \eqref{monotone relation 1} from
Lemma \ref{Lemma for monotonicity}, we obtain 
\begin{align*}
\left\langle (\Lambda(q_1)-\Lambda(q_0))F^{k},F^{k}\right\rangle  & \leq\int_{\Omega}(q_{1}-q_{0})|u_{0}^{k}|^{2}\dx\\
 & \leq\|q_{1}-q_{0}\|_{L^{\infty}(\Omega\backslash M)}\|u_{0}^{(k)}\|_{L^{2}(\Omega\backslash M)}^{2}-\delta\|u_{0}^{k}\|_{L^{2}(M)}^2\\
 &\to-\infty,\text{ as }k\to \infty.
\end{align*}
This shows that $\Lambda(q_0)\not\leq\Lambda(q_1)$. \end{proof}

Theorem~\ref{Theorem equivalent monotonicity} implies global uniqueness for the fractional Calder\'on problem. But let us stress again, that this has already been proven in \cite{ghosh2016calder} and we have used the results from \cite{ghosh2016calder} in our proof of the existence of localized potentials. 
\begin{corollary}
Let $n\in \N$, $\Omega\subseteq\mathbb{R}^{n}$ be a bounded open set, and $0<s<1$. For any two potentials $q_{0},q_{1}\in L_{+}^{\infty}(\Omega)$,
\begin{equation*}
q_{0}= q_{1}\mbox{ if and only if }\Lambda(q_0)=\Lambda(q_1).
\end{equation*}
\end{corollary}
\begin{proof}
This follows immediately from Theorem~\ref{Theorem equivalent monotonicity}.
\end{proof}

Moreover, Theorem~\ref{Theorem equivalent monotonicity} suggests the constructive uniqueness result that $q$ can be reconstructed from
$\Lambda(q)$ by taking the supremum over an appropriate class of test functions $\psi$ with $\Lambda(\psi)\leq \Lambda(q)$. For a rigorous formulation of this result we have to be attentive to the somewhat subtle fact that function values on null sets might still affect the supremum when the supremum is taken over uncountably many functions.

Recall that a point $x\in \mathbb R^n$ is called of \emph{density one} for $E$ if 
\begin{align}\notag
\lim_{r\to 0}\dfrac{|B(x,r)\cap E|}{|B(x,r)|}=1,
\end{align}
where $B(x,r)$ stands for the ball of radius $r$ and centered at $x$. We therefore define the space of \emph{density one simple functions}
\begin{align*}
\Sigma &:=\textstyle \left\{ \psi=\sum_{j=1}^m a_j \chi_{M_j}:\ a_j\in \mathbb{R},\ \text{$M_j\subseteq \Omega$ is a density one set} \right\},
\end{align*}
where we call a subset $M\subseteq \Omega$ a \emph{density one set} if it is non-empty, measurable and has Lebesgue density $1$ in 
all $x\in M$. 
$\Sigma_+\subseteq \Sigma$ denotes the subset of density one simple functions with positive essential infima on $\Omega$ (i.e., where all coefficients $a_j$ are positive and $\Omega\setminus \bigcup_{j=1}^m M_j$ is a null set).

\begin{lemma}\label{lemma:density_one}\begin{enumerate}[(a)]
\item Density one sets have positive measure.
\item Non-empty finite intersections of density one sets are density one sets.
\item If $\psi\in \Sigma$ is nonzero at some point $\hat x\in \Omega$, then there exists a density one set $M$ containing $\hat x$ so that $\psi(x)=\psi(\hat x)$ for all $x\in M$.
\end{enumerate}
\end{lemma}
\begin{proof}
Assertion (a) is obvious. To prove (b), let $M_1,M_2\subseteq \Omega$ be density one sets, and let $x\in M_1\cap M_2$.
Then
\begin{align*}
\lim_{r\to 0}\dfrac{|B(x,r)\setminus (M_1\cap M_2)|}{|B(x,r)|}
\leq \lim_{r\to 0}\dfrac{|B(x,r)\setminus M_1|}{|B(x,r)|} + \lim_{r\to 0}\dfrac{|B(x,r)\setminus M_2|}{|B(x,r)|}=0,
\end{align*}
which shows that
\begin{align*}
\lim_{r\to 0}\dfrac{|B(x,r)\cap M_1\cap M_2|}{|B(x,r)|}=1.
\end{align*}
For the last assertion, let $\psi\in \Sigma$. Then $\psi=\sum_{j=1}^m a_j \chi_{M_j}$ with $a_j\in \mathbb{R}$ and density one sets $M_j\subseteq \Omega$. Then $\psi$ is constant on the intersection of all $M_j$ containing $\hat x$, so that (c) follows from (b).
\end{proof}

Lemma~\ref{lemma:density_one}(c) shows that we might interpret the density one simple functions as simple functions where function values that are only attained on a null set are replaced by zero.
Note also, that the Lebesgue's density theorem implies that every measurable set agrees almost everywhere with a density one set (see Corollary 3 in Section 1.7 of \cite{evans2015measure} for instance), 
and thus every simple function agrees with a density one simple function almost everywhere in $\Omega$. 

As before, we write $\psi\leq q$ if $\psi(x)\leq q(x)$ almost everywhere in $\Omega$. 
We then have the following variant of the simple function approximation theorem.
\begin{lemma}\label{lemma:supremum_simple_functions}
For each function $q\in L^\infty_+(\Omega)$ we have that
\[
q(x)=\sup\{ \psi(x):\ \psi\in \Sigma_+,\ \psi\leq q\} \quad \text{ almost everywhere in } \Omega.
\]
\end{lemma}
\begin{proof}
By the standard simple function approximation theorem \cite{royden1988real}, there exists a sequence
$(\psi_k)_{k\in \mathbb{N}}$ of simple functions with $\psi_k\leq q$ and $\| \psi_k-q \|_{L^\infty(\Omega)}\leq 1/k$.
$q\in L^\infty_+(\Omega)$ implies that $\psi_k\in L^\infty_+(\Omega)$ for almost all $k\in \mathbb{N}$, and by changing the values of the countably many 
functions $\psi_k$ on a null set we can assume that $\psi_k\in \Sigma_+$. This shows that
\[
q(x)=\lim_{k\to \infty} \psi_k (x) \leq \sup\{ \psi(x):\ \psi\in \Sigma_+,\ \psi\leq q\} \quad \text{ almost everywhere in } \Omega.
\]

To show equality, it suffices to show that for each $\delta>0$ the set 
\[
M:=\{ x\in \Omega:\ q(x)+\delta< \sup\{ \psi(x):\ \psi\in \Sigma_+,\ \psi\leq q\} \}
\]
is a null set. To prove this, assume that $M$ is not a null set for some $\delta>0$. 
By removing a null set from $M$, we can assume that $M$ is a density one set and that $q(x)>0$ for all $x\in M$.
By using the Lusin's theorem (see \cite{royden1988real} for instance), all measurable function are approximately continuous at almost every point, $M$ must contain a point
$\hat x$ in which $q$ is approximately continuous, and thus the set
\[
M':=\{ x\in \Omega: q(x)\leq q(\hat x)+\delta / 3 \}
\]
has density one in $\hat x$. (see \cite{evans2015measure}). Removing a null set, we can assume that $M'$ is a density one set still containing $\hat x$.

Moreover, by the definition of $M$, there must exist a $\psi\in \Sigma_+$ with $\psi\leq q$ and
\[
q(\hat x)+\frac{2}{3}\delta \leq \psi(\hat x).
\]
Since $q(\hat x)>0$, there exists a density one set $M''$ containing $\hat x$ where $\psi(x)=\psi(\hat x)$ for all $x\in M''$.

We thus have that 
\[
q(x)+\delta/3 \leq q(\hat x)+\frac{2}{3}\delta \leq \psi(\hat x) = \psi(x)  \quad \text{ for all } x\in M'\cap M'',
\]
with density one sets $M'$ and $M''$ that both contain $\hat x$, so that their intersection possesses positive measure.
But this contradicts that $q(x)\geq \psi(x)$ almost everywhere, and thus shows that $M$ is a null set for all $\delta>0$, and
hence 
\[
q(x)\geq \sup\{ \psi(x):\ \psi\in \Sigma_+,\ \psi\leq q\} \quad \text{ almost everywhere in } \Omega.
\]
\end{proof}

\begin{corollary}
\label{cor:Calderon} Let $n\in\N$, $\Omega\subseteq\mathbb{R}^{n}$
be an open set, and $0<s<1$. A potential $q\in L_{+}^{\infty}(\Omega)$
is uniquely determined by $\Lambda(q)$ via the following
formula
\[
q(x)=\sup\{ \psi(x):\ \psi\in \Sigma_+,\ \Lambda(\psi)\leq \Lambda(q)\} \quad \text{ almost everywhere in } \Omega.
\]
\end{corollary}
\begin{proof}
This follows immediately from Theorem~\ref{Theorem equivalent monotonicity} and Lemma~\ref{lemma:supremum_simple_functions}.
\end{proof}

\section{Shape reconstruction by linearized monontonicity tests \label{Section 5}}

The results in Section \ref{section:calderon} show that the coefficient $q$ in the fractional Schr\"odinger equation
\[
(-\Delta)^s u + qu = 0 \quad \text{ in } \Omega
\]
can be reconstructed from the Dirichlet-to-Neumann operator $\Lambda(q)$ by comparing $\Lambda(q)$ with the DtN map $\Lambda(\psi)$
of (density one) simple functions $\psi$. A practical implementation of these monotonicity tests would require solving the fractional Schr\"odinger equation for each utilized simple function $\psi$. 

In this section we will study the shape reconstruction problem of determining regions where a coefficient function $q\in L^\infty_+(\Omega)$ changes from 
a known reference function $q_0\in L^\infty_+(\Omega)$ (e.g., $q_0$ may describe a background coefficient, and $q_1$ denotes the coefficient function 
in the presence of anomalies or scatterers). We will show that the support of $q_1-q_0$ can be reconstructed with
\emph{linearized monotonicity tests} \cite{harrach2013monotonicity,garde2017comparison}. These linearized tests only utilize the solution of the fractional Schr\"odinger equation
with the reference coefficient function $q_0\in L^\infty_+(\Omega)$. They do not require any other special solutions of the equation.

%

\subsection{Linerization of the Dirichlet-to-Neumann operator}

We start by showing Fr\'echet differentiability of the Dirichlet-to-Neumann operator.

\begin{lemma}\label{lemma:Frechet_derivative}
Let $n\in \N$, $\Omega\subseteq\mathbb{R}^{n}$ be a bounded open set, and $0<s<1$. The Dirichlet-to-Neumann operator
\[
\Lambda:\ \mathscr{D}(\Lambda):=L^\infty_+(\Omega)\subset L^\infty(\Omega)\to \mathcal{L}(H(\Omega_{e}),H(\Omega_{e})^*),
\quad q\mapsto \Lambda(q),
\]
is Fr\'echet differentiable. At $q\in L^\infty_+(\Omega)$ its derivative is given by
\begin{align*}
\Lambda'(q):& \ L^\infty(\Omega) \to \mathcal{L}(H(\Omega_{e}),H(\Omega_{e})^*),\quad r\mapsto \Lambda'(q)r,\\
\left\langle (\Lambda'(q) r) F, G \right\rangle:&=\int_{\Omega} r S_q(F) S_q(G) dx \quad \text{ for all } r\in L^\infty(\Omega),\ F,G\in H(\Omega_{e}),
\end{align*}
where $S_q:\ H(\Omega_{e})\to H^s(\mathbb R^n)$, $F\mapsto u$,
is the solution operator of the Dirichlet problem
\[
(-\Delta)^{s}u + q u=0  \mbox{ in }\Omega \quad\text{ and } \quad u|_{\Omega_{e}}=F.
\]
\end{lemma}
\begin{proof}
Let $q\in L^\infty_+(\Omega)$. $\Lambda'(q)$ is a linear bounded operator since $S_q$ is linear and bounded, cf.~Lemma \ref{lemma:forward_theory}. 
For sufficiently small $r\in L^\infty(\Omega)$, so that $q+r\in L^\infty_+(\Omega)$, we obtain from the monotonicity relations \eqref{monotone relation 1} and \eqref{monotone relation 3} in Lemma \ref{Lemma for monotonicity} that for all $F\in H(\Omega_{e})$,
\begin{align*}
0 \geq \left\langle \left(\Lambda(q+r)-\Lambda(q)-\Lambda'(q)r\right) F,F\right\rangle  
\geq \int_{\Omega}\left( \dfrac{q}{q+r}r - r\right) |u_{q}|^{2} dx,
\end{align*}
where $u_q=S_q(F)$.

Using that $\Lambda(q)$, $\Lambda(q+r)$, and $\Lambda'(q)r$ are symmetric operators, it follows that
\begin{align*}
&\lefteqn{ \| \Lambda(q+r)-\Lambda(q)-\Lambda'(q)r\|_{\mathcal{L}(H(\Omega_{e}),H(\Omega_{e})^*)} }\\
=& \sup_{\| F\|_{H(\Omega_{e})} =1 } 
\left| \left\langle \left(\Lambda(q+r)-\Lambda(q)-\Lambda'(q)r\right) F,F\right\rangle \right|\\
 \leq &\sup_{\| F\|_{H(\Omega_{e})} =1 } \int_{\Omega}\left| \dfrac{q}{q+r}r - r\right| |u_{q}|^{2} dx
\leq \left\| \frac{r^2}{q+r} \right\|_{L^\infty(\Omega)} \sup_{\| F\|_{H(\Omega_{e})} =1 } \| S_q(F) \|_{L^2(\Omega)}^2\\
\leq &\left\| r \right\|_{L^\infty(\Omega)} \left\| \frac{r}{q+r} \right\|_{L^\infty(\Omega)} \| S_q\|_{\mathcal L(H(\Omega_{e}),H^s(\mathbb R^n)) },
\end{align*}
which shows 
\begin{equation*}
\lim_{\| r\|_{L^\infty(\Omega)} \to 0} \frac{\| \Lambda(q+r)-\Lambda(q)-\Lambda'(q)r\|_{\mathcal{L}(H(\Omega_{e}),H(\Omega_{e})^*)}}{\| r \|_{L^\infty(\Omega)}}=0.
\end{equation*}
\end{proof}

\begin{remark}\label{remark:monotonicity_with_Frechet}
Using the Fr\'echet derivative, the monotonicity relations \eqref{monotone relation 1} and \eqref{monotone relation 3} in Lemma \ref{Lemma for monotonicity}
can be written as follows. For all $q_0,q_1 \in L^\infty_+(\Omega)$
\begin{align*}
\Lambda'(q_0)(q_1-q_0) \geq \Lambda(q_1)-\Lambda(q_{0})\geq \Lambda'(q_0)\left( \frac{q_0}{q_1} (q_1-q_0) \right).
\end{align*}
\end{remark}

We also have an analogue of the monotonicity result in Theorem \ref{Theorem equivalent monotonicity}.
\begin{theorem}\label{thm:converse_mon_frechet}
Let $n\in \N$, $\Omega\subseteq\mathbb{R}^{n}$ be a bounded open set, and $0<s<1$. Then for all $q\in L^\infty_+(\Omega)$ and $r_0,r_1\in L^\infty(\Omega)$,
\[
r_0\leq r_1 \quad \text{ if and only if } \quad \Lambda'(q)r_0\leq \Lambda'(q)r_1.
\]
\end{theorem}
\begin{proof}
If $r_0\leq r_1$ then $\Lambda'(q)r_0\leq \Lambda'(q)r_1$ follows immediately from the characterization 
of $\Lambda'(q)$ in Lemma \ref{lemma:Frechet_derivative}. The converse follows from the same localized potentials argument
as in the proof of Theorem \ref{Theorem equivalent monotonicity}.
\end{proof}

Note that this implies uniqueness of the linearized fractional Calder\'on problem:
\begin{corollary}
Let $n\in \N$, $\Omega\subseteq\mathbb{R}^{n}$ be a bounded open set, and $0<s<1$. 
For all $q\in L^\infty_+(\Omega)$, the Fr\'echet derivative $\Lambda'(q)$ is injective, i.e.
\begin{equation*}
\Lambda'(q) r=0 \quad \text{ if and only if } \quad r=0.
\end{equation*}
\end{corollary}
\begin{proof}
This follows immediately from Theorem~\ref{thm:converse_mon_frechet}.
\end{proof}

\subsection{Reconstructing the support of a coefficient change}

In this subsection, let $n\in \N$, $\Omega\subseteq\mathbb{R}^{n}$ be a bounded open set, and $0<s<1$.
As in the introduction, let $q_0\in L^\infty_+(\Omega)$ denote a known reference coefficient, and $q_1\in L^\infty_+(\Omega)$
denote an unknown coefficient function that differs from the reference value $q_0$ in certain regions. 
We aim to find these anomalous regions (or scatterers), i.e., the support of $q_1-q_0$, from the difference of
the Dirichlet-to-Neumann-operators $\Lambda(q_1)-\Lambda(q_0)$.

To that end, we introduce, for a measurable subset $M\subseteq \Omega$, the testing operator 
$\mathcal{T}_{M}:H(\Omega_{e})\to H(\Omega_{e})^{*}$ by 
setting $T_M:=\Lambda'(q_0)\chi_M$. i.e.,
\begin{equation}
\label{TB operator}
\left\langle \mathcal{T}_{M}F,G\right\rangle :=\int_{M} S_{q_0}(F) S_{q_0}(G) dx
\quad \text{ for all } F,G\in H(\Omega_{e}),
\end{equation}
where, as in Lemma~\ref{lemma:Frechet_derivative}, $S_{q_0}:\ H(\Omega_{e})\to H^s(\mathbb R^n)$, $F\mapsto u_0$,
denotes the solution operator of the reference Dirichlet problem
\[
(-\Delta)^{s}u_0 + q_0 u_0=0  \mbox{ in }\Omega \quad\text{ and } \quad u_0|_{\Omega_{e}}=F.
\]

The following theorem shows that we can find the support of $q-q_0$ by shrinking closed sets, cf.\ \cite{harrach2013monotonicity,garde2019regularized}.
\begin{theorem}\label{thm:support_from_closed_sets}
For each closed subset $C\subseteq \Omega$, 
\begin{align*}
\mathrm{supp}(q_1-q_0)\subseteq C \quad \text{ if and only if } \quad \exists \alpha>0:\ 
-\alpha \mathcal{T}_C \leq \Lambda(q_1)-\Lambda(q_0)\leq  \alpha \mathcal{T}_C.
\end{align*}
Hence,
\[
\mathrm{supp}(q_1-q_0)=\bigcap \{ C\subseteq \Omega \text{ closed}:\ \exists \alpha>0:\ 
-\alpha \mathcal{T}_C \leq \Lambda(q_1)-\Lambda(q_0)\leq  \alpha \mathcal{T}_C\}.
\]
\end{theorem}
\begin{proof}
\begin{enumerate}
\item[(a)] Let $\mathrm{supp}(q_1-q_0)\subseteq C$. Then every sufficiently large $\alpha>0$ 
fulfills 
\[
q_1\leq q_0 + \alpha \chi_C.
\]
Using Theorem \ref{Theorem equivalent monotonicity} and Remark \ref{remark:monotonicity_with_Frechet}, we thus obtain
\begin{align*}
\Lambda(q_1)\leq \Lambda(q_0 + \alpha \chi_C)\leq \Lambda(q_0)+\Lambda'(q_0)\alpha \chi_C = \Lambda(q_0) + \alpha \mathcal{T}_C.
\end{align*}
Moreover, for sufficiently small $\beta>0$ we also have that
\[
q_1\geq q_0 + (\beta-q_0) \chi_C \quad \text{ and } \quad q_0\geq \beta
\] 
and thus (using Theorems \ref{Theorem equivalent monotonicity}, \ref{thm:converse_mon_frechet}, and Remark \ref{remark:monotonicity_with_Frechet})
\begin{align*}
\Lambda(q_1)-\Lambda(q_0)&\geq \Lambda(q_0 + (\beta-q_0) \chi_C)-\Lambda(q_0)\\
&\geq \Lambda'(q_0)\left( \frac{q_0}{q_0 + (\beta-q_0) \chi_C} (\beta-q_0) \chi_C\right)\\
&\geq - \Lambda'(q_0)\left( \frac{q_0^2}{q_0 + (\beta-q_0) \chi_C} \chi_C\right)\\
&\geq -\dfrac{1}{\beta } \| q_0 \|^2_{L^\infty(\Omega)} \Lambda'(q_0) \chi_C,
\end{align*}
which shows that 
\[
\Lambda(q_1)-\Lambda(q_0)\geq -\alpha \mathcal{T}_C
\]
is also fulfilled for sufficiently large $\alpha>0$.
\item[(b)] To show the converse implication, let $\alpha>0$ fulfill 
\[
-\alpha \mathcal{T}_C \leq \Lambda(q_1)-\Lambda(q_0)\leq  \alpha \mathcal{T}_C.
\]
Then we obtain using Remark \ref{remark:monotonicity_with_Frechet}
\begin{align*}
\Lambda'(q_0)(-\alpha\chi_C)&=-\alpha \mathcal{T}_C \leq \Lambda(q_1)-\Lambda(q_0)\leq \Lambda'(q_0)(q_1-q_0),
\end{align*}
so that it follows from Theorem \ref{thm:converse_mon_frechet} that
\[
q_1-q_0\geq -\alpha\chi_C,
\]
and in particular $q_1-q_0\geq 0$ almost everywhere on $\Omega\setminus C$.

Likewise we obtain using Remark \ref{remark:monotonicity_with_Frechet}
\begin{align*}
\Lambda'(q_0)(\alpha\chi_C)&=\alpha \mathcal{T}_C \geq \Lambda(q_1)-\Lambda(q_0)\geq \Lambda'(q_0)\left( \frac{q_0}{q_1} (q_1-q_0) \right),
\end{align*}
so that it follows from Teorem \ref{thm:converse_mon_frechet} that
\[
\frac{q_0}{q_1} (q_1-q_0)\leq \alpha\chi_C.
\]
Since $q_1,q_0\in L^\infty_+(\Omega)$ this yields that $q_1-q_0\leq 0$ almost everywhere on $\Omega\setminus C$.
Hence, $q_1=q_0$ almost everywhere in the open set $\Omega\setminus C$ and thus $\mathrm{supp}(q_1-q_0)\subseteq C$.
\end{enumerate}
\end{proof}

In the definite case that either $q_1\geq q_0$ or $q_1\leq q_0$ holds almost everywhere in $\Omega$, 
we can also use the union of small test balls to characterize the so-called \emph{inner support} of $q_1-q_0$. The inner support $\mathrm{inn\,supp}(r)$ of a measurable function $r:\ \Omega\to \mathbb{R}$ is defined as the union of all open sets $U$ on which the essential infimum of $|\kappa|$ is positive, cf.\ \cite[Section 2.2]{harrach2013monotonicity}. 

\begin{theorem}\label{thm:support_from_open_balls}
\begin{enumerate}
\item[(a)] Let $q_1\leq q_0$. For every open set $B\subseteq \Omega$ and every $\alpha>0$
\begin{enumerate}
\item[(1)] $q_1\leq q_0-\alpha \chi_B$ implies $\Lambda(q_1) \leq \Lambda(q_0)-\alpha\mathcal{T}_{B}$.
\item[(2)] $\Lambda(q_1) \leq \Lambda(q_0)-\alpha\mathcal{T}_{B}$ implies $B\subseteq \mathrm{inn\,supp}(q_1-q_0)$.
\end{enumerate}
Hence,
\begin{align*}
\mathrm{inn\,supp}(q_1-q_0)=
\bigcup \{B\subseteq \Omega \text{ open ball}:\ \exists \alpha>0: \Lambda(q_1) \leq \Lambda(q_0)-\alpha\mathcal{T}_{B}\}.
\end{align*}
\item[(b)] Let $q_1\geq q_0$. For every open set $B\subseteq \Omega$ and every $\alpha>0$
\begin{enumerate}
\item[(1)] $q_1\geq q_0+\alpha \chi_B$ implies $\Lambda(q_1) \geq \Lambda(q_0)+\tilde\alpha\mathcal{T}_{B}$ 
with $\tilde\alpha:=\frac{\inf(q_0)\alpha}{\inf(q_0)+\alpha}$.
\item[(2)] $\Lambda(q_1) \geq \Lambda(q_0)+\alpha\mathcal{T}_{B}$ implies $B\subseteq \mathrm{inn\,supp}(q-q_0)$.
\end{enumerate}
Hence,
\begin{align*}
\mathrm{inn\,supp}(q_1-q_0)=\bigcup \{B\subseteq \Omega \text{ open ball}:\ \exists \alpha>0: \Lambda(q_1) \geq \Lambda(q_0)+\alpha\mathcal{T}_{B}\}.
\end{align*}
\end{enumerate}
\end{theorem}
\begin{proof}
\begin{enumerate}
\item[(a)] If $q_1\leq q_0-\alpha \chi_B$, then we obtain using Theorem \ref{thm:converse_mon_frechet}, and Remark \ref{remark:monotonicity_with_Frechet} that
\begin{align*}
\Lambda(q_1)-\Lambda(q_0)&\leq \Lambda'(q_0)(q_1-q_0)\leq -\alpha \Lambda'(q_0)\chi_B=-\alpha \mathcal{T}_{B}.
\end{align*}
On the other hand, if $\Lambda(q) \leq \Lambda(q_0)-\alpha\mathcal{T}_{B}$ then we obtain from Remark \ref{remark:monotonicity_with_Frechet} that
\begin{align*}
-\alpha \Lambda'(q_0)\chi_B=-\alpha\mathcal{T}_{B}\geq \Lambda(q_1)-\Lambda(q_0)\geq 
\Lambda'(q_0)\left( \frac{q_0}{q_1} (q_1-q_0) \right)
\end{align*}
so that it follows from Theorem \ref{thm:converse_mon_frechet} that
\[
-\alpha \chi_B \geq \frac{q_0}{q_1} (q_1-q_0).
\]
Hence, $q_0-q_1\geq \frac{\inf(q_1)}{\sup(q_0)}\alpha$ almost everywhere on $B$ and thus $B\subseteq \mathrm{inn\,supp}(q_1-q_0)$.
\item[(b)] If $q_1\geq q_0+\alpha \chi_B$, then we obtain using Theorems \ref{Theorem equivalent monotonicity}, \ref{thm:converse_mon_frechet}, and Remark \ref{remark:monotonicity_with_Frechet} that
\begin{align*}
\Lambda(q_1)-\Lambda(q_0)&\geq \Lambda(q_0+\alpha \chi_B)-\Lambda(q_0)\\
&\geq \Lambda'(q_0)\left( \frac{q_0}{q_0+\alpha \chi_B} \alpha \chi_B\right)= \Lambda'(q_0)\left(\left( 1 - \frac{\alpha}{q_0+\alpha }\right) \alpha \chi_B\right)\\
&\geq \Lambda'(q_0)\left(\left( 1 - \frac{\alpha }{\inf(q_0)+\alpha }\right) \alpha \chi_B\right)= \frac{\inf(q_0)\alpha}{\inf(q_0)+\alpha}\mathcal{T}_{B}.
\end{align*}
On the other hand, if $\Lambda(q_1) \geq \Lambda(q_0)+\alpha\mathcal{T}_{B}$ then we obtain from Remark \ref{remark:monotonicity_with_Frechet} that
\begin{align*}
\alpha \Lambda'(q_0)\chi_B=\alpha\mathcal{T}_{B}\leq \Lambda(q_1) - \Lambda(q_0) \leq \Lambda'(q_0)(q_1-q_0),
\end{align*}
so that it follows from Theorem \ref{thm:converse_mon_frechet} that
\[
\alpha \chi_B \leq q_1-q_0,
\]
and thus $B\subseteq \mathrm{inn\,supp}(q_1-q_0)$.
\end{enumerate}
\end{proof}

\section{Discussion and Outlook}

We have shown an if-and-only-if monotonicity relation between a positive potential $q\in L_+^\infty(\Omega)$ in the fractional Schr\"odinger equation, and the associated Dirichlet-to-Neumann operator $\Lambda(q)$ (cf.~Theorem~\ref{Theorem equivalent monotonicity})
\begin{equation*}
q_{0}\leq q_{1}\mbox{ if and only if }\Lambda(q_0)\leq\Lambda(q_1).
\end{equation*}
From this we obatined a constructive uniqueness result for the Calderón problem for
the fractional Schrödinger equation. The potential is uniquely determined by the simple reconstruction formula
(cf.~Corollary~\ref{cor:Calderon})
\begin{equation*}
q(x)=\sup\left\{ \psi(x):\mbox{ $\psi$ positive (density one) simple function, } \Lambda(\psi)\leq\Lambda(q)\right\}.
\end{equation*}

Let us give some remarks on a possible practical implementation of our results. First of all, let us stress that
the localized potentials used in this work can be created with Dirichlet data supported in arbitrarily small open subsets $\emptyset\neq O\subseteq \Omega_e$, cf.~Corollary~\ref{cor:localized_potentials}. Hence, all results in this work remain valid if the full data DtN is replaced by the partial data DtN
\[
\Lambda(q):\ H_0^s(O)\to H^{-s}(O),
\]
where $H_0^s(O)$ is the closure of $C_c^\infty(O)$ in $H^s(\R^n)$, and $H^{-s}(O):=H_0^s(O)'$. 

For a numerical implementation, one could choose a family of characteristic functions $\chi_1,\ldots,\chi_M$ for disjoint density one sets (e.g., a pixel partition) $P_1,\ldots,P_M\subseteq \Omega$, $M\in \N$, 
and determine 
\[
\alpha_m:=\sup\{ \alpha\in \R:\ \Lambda(\alpha_m \chi_m)\leq \Lambda(q)\}.
\]
Then, $\psi=\sum_{m=1}^M \alpha_m \chi_m$ is the largest piecewise-constant function (on the given partition) with $\psi\leq q$. 
Analogously, one could obtain a piecewise-constant function approximation $q$ from above.
A numerical implementation of this approach would be computationally rather expensive as it requires solving the fractional Schrödinger equation for a large number of sets $P_m$ and contrast levels $\alpha$ (though these solutions could be precomputed in advance).

A computationally more efficient approach can be used for detecting regions where the potential $q$ differs from a known reference function $q_0$. The support of this change can be determined by shrinking closed sets according to the formula (cf.~Theorem~\ref{thm:support_from_closed_sets})
\[
\mathrm{supp}(q-q_0)=\bigcap \{ C\subseteq \Omega \text{ closed}:\ \exists \alpha>0:\ 
-\alpha \mathcal{T}_C \leq \Lambda(q)-\Lambda(q_0)\leq  \alpha \mathcal{T}_C\},
\]
where the operator $\mathcal{T}_C$ can be calculated from integrating the solution of the fractional Schr\"odinger equation for the reference potential $q_0$ over the set $C$, and no other PDE solutions are required for this approach. Moreover, the inner support of the potential change can be calculated by comparing $\Lambda(q)-\Lambda(q_0)$ with $T_B$ for open balls $B$, cf.~Theorem~\ref{thm:support_from_open_balls}. 

Algorithms based on linearized monotonicity tests have been successfully applied to the standard Laplacian case ($s=1$), cf.~the works cited in the introduction. Among these works, let us mention the recent papers 
\cite{garde2017convergence,harrach2018monotonicity} that show reconstructions on simulated and real-life measurement data,
and discuss practical implementation issues and the regularization of measurement errors. 

For the standard Laplacian case, monotonicity-based reconstruction methods have recently been extended to the Schr\"odinger (or Helmholtz) equation with general (not-necessarily positive) potential function $q\in L^\infty(\Omega)$, cf.~\cite{harrach2018monotonicity,griesmaier2018monotonicity}, and monotonicity arguments were also used to prove stability results, cf.~\cite{harrach2019global,harrach2019uniqueness,seo2018learning,eberle2019lipschitz}.
The recent follow-up paper \cite{harrach2019monotonicity} extends the results to general potentials $q\in L^\infty(\Omega)$
in the fractional diffusion case and proves Lipschitz stability with finitely many measurements.

\section*{Acknowledgment}

Yi-Hsuan Lin is partially supported by MOST of Taiwan under the project
160-2917-I-564-048.

\bibliographystyle{abbrv}
\bibliography{ref}

\end{document}